\documentclass[12pt]{amsart}

\usepackage{amsthm}
\usepackage{amsmath}
\usepackage{amssymb}
\usepackage{eucal}
\usepackage{amsfonts}
\usepackage{amsmath}
\usepackage[Lenny]{fncychap}
\usepackage{enumerate}
\usepackage{amssymb}
\usepackage[english]{babel}
\usepackage[latin1]{inputenc}
\usepackage[breaklinks=true]{hyperref}
\usepackage[all]{xy}
\usepackage{fancybox}
\usepackage{latexsym,mathrsfs,longtable,lscape,dsfont,yfonts}
\usepackage{srcltx}

\ifx\pdfpagewidth\undefined 
\usepackage[T1]{fontenc}
\else 
\usepackage[OT1]{fontenc} 
\fi 
\usepackage{amsfonts,amssymb,latexsym,longtable,amsmath}

\newtheorem{thm}{Theorem}[section]
\newcommand{\bthm}{\begin{thm}} \newcommand{\ethm}{\end{thm}}
\newtheorem{prop}[thm]{Proposition}
\newcommand{\bprp}{\begin{prop}} \newcommand{\eprp}{\end{prop}}
\newtheorem{fact}[thm]{Fact}
\newcommand{\bfct}{\begin{fact}} \newcommand{\efct}{\end{fact}}
\newtheorem{prob}[thm]{Problem}
\newcommand{\bprb}{\begin{prob}} \newcommand{\eprb}{\end{prob}}
\newtheorem{quest}[thm]{Question}
\newcommand{\bqtn}{\begin{quest}} \newcommand{\eqtn}{\end{quest}}
\newtheorem{lem}[thm]{Lemma}
\newcommand{\blem}{\begin{lem}} \newcommand{\elem}{\end{lem}}
\newtheorem{claim}[thm]{Claim}
\newcommand{\bclm}{\begin{claim}} \newcommand{\eclm}{\end{claim}}
\newtheorem{cor}[thm]{Corollary}
\newcommand{\bcor}{\begin{cor}} \newcommand{\ecor}{\end{cor}}
\newtheorem{conj}[thm]{Conjecture}
\newcommand{\bcnj}{\begin{conj}} \newcommand{\ecnj}{\end{conj}}
\theoremstyle{definition}
\newtheorem{defn}[thm]{Definition}
\newcommand{\bdfn}{\begin{defn}} \newcommand{\edfn}{\end{defn}}
\newtheorem{spec}[thm]{Specializing}
\newcommand{\bspc}{\begin{spec}} \newcommand{\espc}{\end{spec}}
\theoremstyle{remark}
\newtheorem{rem}[thm]{Remark}
\newcommand{\brem}{\begin{rem}} \newcommand{\erem}{\end{rem}}
\newtheorem{cnv}[thm]{Convention}
\newcommand{\bcnv}{\begin{cnv}} \newcommand{\ecnv}{\end{cnv}}
\newtheorem{exam}[thm]{Example}
\newcommand{\bexm}{\begin{exam}} \newcommand{\eexm}{\end{exam}}
\newcommand{\bpf}{\begin{proof}} \newcommand{\epf}{\end{proof}}

\newcommand{\Z}{\mathbb Z}

\newcommand{\N}{\mathbb N}

\newcommand{\sA} {{\mathcal A}}
\newcommand{\sB} {{\mathcal B}}
\newcommand{\sC} {{\mathcal C}}
\newcommand{\sD} {{\mathcal D}}
\newcommand{\sE} {{\mathcal E}}

\newcommand{\sG} {{\mathcal G}}

\newcommand{\sO} {{\mathcal O}}

\newcommand{\sS} {{\mathcal S}}

\renewcommand{\phi}{\varphi}
\renewcommand{\theta}{\vartheta}

\newcommand{\gu}{{\upsilon}}

\newcommand{\mkp}{\medskip}

\def\defi{\buildrel\rm def \over=}

\begin{document}

\title[Representation of Group Isomorphisms]{Representation of Group Isomorphisms I}

\author[M. Ferrer]{Marita Ferrer}
\address{Universitat Jaume I, Instituto de Matem\'aticas de Castell\'on,
Campus de Riu Sec, 12071 Castell\'{o}n, Spain.}
\email{mferrer@mat.uji.es}

\author[M. Gary]{Margarita Gary}
\address{Departamento de Ciencias Naturales y Exactas, Universidad de la Costa,  CUC,
 Calle 58, 55-66, Barranquilla, Colombia.}
\email{mgary@cuc.edu.co}

\author[S. Hern\'andez]{Salvador Hern\'andez}
\address{Universitat Jaume I, INIT and Departamento de Matem\'{a}ticas,
Campus de Riu Sec, 12071 Castell\'{o}n, Spain.}
\email{hernande@mat.uji.es}

\thanks{ Research Partially supported by the Spanish Ministerio de Econom\'{i}a y Competitividad,
grant MTM2016-77143-P (AEI/FEDER, EU),
and the Universitat Jaume I, grant P1171B2015-77.}

\begin{abstract}
Let $G$ be a metric group and let $\sA ut(G)$ denote the automorphism group of $G$. If $\sA$ and $\sB$ are
groups of $G$-valued maps defined on the sets $X$ and $Y$, respectively,
we say that $\sA$ and $\sB$ are \emph{equivalent} if there is a group isomorphism
$H\colon\sA\to\sB$ such that there is a bijective map $h\colon Y\to X$ and a map $w\colon Y\to \sA ut (G)$
satisfying  $Hf(y)=w[y](f(h(y)))$ for all $y\in Y$ and $f\in \sA$.
In this case, we say that $H$ is represented as a \emph{weighted composition
operator}. A group isomorphism
$H$ defined between $\sA$ and $\sB$ is called \emph{separating} when for each pair of maps
$f,g\in \sA$ satisfying that $f^{-1}(e_G)\cup g^{-1}(e_G)=X$, it holds that $(Hf)^{-1}(e_G)\cup (Hg)^{-1}(e_G)=Y$.
Our main result establishes that under some mild conditions, every separating group isomorphism can be represented as a weighted
composition operator. As a consequence we establish the equivalence of
two function groups if there is a biseparating isomorphism defined between them.
\end{abstract}



\date{15 May 2018}

\maketitle \setlength{\baselineskip}{24pt}

\section {Introduction}
There are many results that are concerned with the representation of linear operators as weighted composition
maps and the equivalence of specific groups of continuous functions in the literature, which is
vast in this regard. We will only mention here the classic Banach-Stone Theorem that, when $G$
is the field of real or complex numbers, establishes that if the Banach spaces of continuous
functions $C(X,G)$ and $C(Y,G)$ are isometric, then they are equivalent and the isometry can be
represented as a weighted composition map
(cf. \cite{Banach,gau-jeang-wong,SHdz:Houston,her-bec-nar,Stone}).
Another important example appears in coding theory, where the well known MacWilliams Equivalence
Theorem asserts that, when $G$ is a finite field and $X$ and $Y$ are finite sets, two codes
(linear subspaces) $\sA$ and $\sB$ of $G^X$ and $G^Y$, respectively, are equivalent when they are
isometric for the Hamming metric (see \cite{McW:ii}).
This result has been generalized to convolutional codes in \cite{GluLue:2009}. 
The main motivation of this research has been to extend MacWilliams Equivalence Theorem to more general
settings and explore the possible application of these methods to the study of convolutional codes
or linear dynamical systems. In order to do this, we apply topological methods as a main tool here and
we will look at the possible application of this
abstract approach elsewhere.

Let $G$ be a metric group and let $\sA ut(G)$ denote the automorphism group of $G$.
Two groups of $G$-valued maps $\sA$ and $\sB$, defined on the sets $X$ and $Y$, respectively,
are \emph{compatible} if there is a group isomorphism
$H\colon\sA\to\sB$ such that there are two maps $h\colon Y\to X$ and $w\colon Y\to \sA ut (G)$
satisfying  $Hf(y)=w[y](f(h(y)))$ for all $y\in Y$ and $f\in \sA$.
In this case, we say that $H$ is represented as a \emph{weighted composition
operator}. Two compatible groups of $G$-valued maps $\sA$ and $\sB$ are said to be \emph{equivalent}
if the map $h\colon Y\to X$ is a bijection. It is not hard to verify that this notion defines an
equivalence relation on the set of all  groups of $G$-valued maps which are defined on an arbitrary but fixed set $X$.

A group isomorphism $H$ defined between two function groups $\sA$ and $\sB$ is called \emph{separating} when for each pair of maps
$f,g\in \sA$ satisfying that $f^{-1}(e_G)\cup g^{-1}(e_G)=X$, it holds that $(Hf)^{-1}(e_G)\cup (Hg)^{-1}(e_G)=Y$.
Our main result establishes that under some mild conditions, every separating group isomorphism can be represented as a weighted
composition operator. As a consequence we establish the equivalence of
two function groups if there is a biseparating isomorphism defined between them.

There are many precedents in the study the representation of group
homomorphisms for group-valued continuous functions. Among them, the following ones are relevant here
(cf \cite{cordeiro,eda kiyosa ohta,Fer_Her_Rod,her_rod:2007,Fer_Her_Gar:jmaa,martinez,ohta 96,Sha_Spe:2010,yang-1,yang}).
Most basic facts and notions related to topological properties may be
found in \cite{gill-jer,engel}.

\mkp

\section{Basic Facts}
Given  $T$, a topological space, and $A\subseteq T$, the symbols $\overline{A}$ and $int A$ will denote the standard closure and interior operators respectively.
Throughout this paper $G$ denotes a metric group, $e_G$ its neutral element and $X$ an infinite set.

\bdfn
Let $G^X$ be the set of all maps from $X$ into $G$, which is a group with the usual pointwise product as composition law.
Given $f\in G^X$, define $Z(f)=\{x\in X:f(x)=e_G\}$ and $coz(f)=X\setminus Z(f)$.
It is said that a subgroup $\sA$ of $G^X$ \emph{separates} points in $X$ if for any two distinct points $x_1,x_2\in X$ there is $f\in \mathcal{A}$ such that $f(x_1)\neq f(x_2)$ and
it is said that $\sA$ is \emph{faithful} if $X=\cup \{ coz(f) : f\in\sA \}$. A \emph{function group} in $G^X$ is a faithful subgroup $\sA$ of $G^X$
that separates the points in $X$.  As a consequence, given a function group $\sA$ in $G^X$, the set $X$ is equipped
with a Hausdorff, completely regular topology $\tau_{\mathcal{A}}$, which has the family $\{f^{-1}(U): U\text{ open in }G, f\in \mathcal{A}\}$
as an open subbase. Therefore $\mathcal{A}$ is a subgroup of $C(X,G)$, the group of all continuous functions from $X$ into $G$,
when $X$ is equipped with the $\tau_{\mathcal{A}}$-topology.

The following collection of subsets of $X$ will be essential in what follows:
(1) $\sC(\sA)=\{f^{-1}(C):f\in \sA,\ C\text{ is a closed subset of }G\}$;
(2) $\sO(\sA)=\{X\setminus Z: Z\in \sC(\sA)\}$; (3) $\mathcal{D}(\sA)\defi$ the minimum collection of subsets that contains $\sC(\sA)$
that is closed under finite unions and intersections (that is, the $\sigma$-ring of sets
generated by $\sC(\sA)$); (4) $\mathcal{E}(\sA)\defi$ the minimum collection of subsets that contains $\sO(\sA)$ that is closed under finite unions and intersections
(that is, the $\sigma$-ring of sets
generated by $\sO(\sA)$).
\edfn

\bdfn
A map $f\in\sA$ is said to be \emph{bounded} if $f(X)$ is relatively compact in $G$. The function group $\sA$ in $G^X$ is
\emph{bounded} if every $f\in\sA$ is bounded.
In general, if $\sA$ is a function group in $G^X$, the set $\sA^*\defi\{f\in \sA : f\ \text{is bounded}\}$ is the largest bounded subgroup of $\sA$.

Let $\sA$ be function subgroup of $G^X$ and let $\kappa$ be an ordinal number. We say that $\sA$ is a \emph{$\kappa$-extension} of $\sA^*$ if
the following assertions hold:
\begin{enumerate}
\item[(i)] $\sD(\sA) = \sD(\sA^*)$.
\item[(ii)] $f\in\sA$ if there is $\{f_i:i<\kappa\}\subseteq\sA^*$,
such that $\{coz(f_i): i<\kappa\}$ is locally finite and $f=\prod\limits_{i<\kappa}f_i$.  
\end{enumerate}

In general, we say that  $\sA$ is an \emph{extension} of $\sA^*$ if it a $\kappa$-extension for some undetermined cardinal $\kappa$.
Remark that if $\sA$ is a function group in $G^X$ that is an extension of $\sA^*$, then the latter is also a function group in $G^X$
and furthermore $\tau_{\sA}=\tau_{\sA^*}$.
\edfn

\section{Bounded function groups}
\bdfn\label{compactification}
Let $\sA$ be a bounded function group in $G^X$.  Then
the Cartesian product $\prod\limits_{f\in\mathcal{A}}\overline{f(X)}$, equipped with product topology
as a subspace of $G^{\mathcal{A}}$, is a compact Hausdorff topological space.
Therefore, the diagonal map $$\varepsilon=\Delta_{f\in\sA}:X\rightarrow \prod\limits_{f\in\sA}\overline{f(X)}$$
$$\varepsilon(x)=(f(x))_{f\in\sA}\ \text{for all}\ x\in X$$
\noindent injects $X$ into $\prod\limits_{f\in\mathcal{A}}\overline{f(X)}$ and $\pi_f\circ\varepsilon=f$ for all $f\in\sA$, where
$\pi_f$ denotes the canonical projection map from $G^{\sA}$ into $G$. The initial topology that $\varepsilon$ defines on $X$
(that is, the product topology) coincides with the \emph{$\tau_{\sA}$-topology} on $X$ defined in the paragraph above.

Set $\beta_{\sA}X\defi\overline{\varepsilon(X)}\subseteq \prod\limits_{f\in\sA}\overline{f(X)}$.
The pair $(\beta_{\sA}X,\varepsilon)$ is a Hausdorff compactification of $(X,\tau_{\sA})$, which has the following property:
for each $f\in \sA$ the map ${\pi_f}_{|\varepsilon(X)}:\varepsilon(X)\rightarrow f(X)$ has a unique continuous extension
$f^b:\beta_{\sA}X\rightarrow\overline{f(X)}$ such that
${f^b}_{|\varepsilon(X)}={\pi_f}_{|\varepsilon(X)}$.
Since ${\pi_f}_{|(\beta_{\sA}X)}$ satisfies this equality, we have that $f^b={\pi_f}_{|(\beta_{\sA}X)}$ and $f=f^b\circ\varepsilon$.

Set $\mathcal{A}^b\defi\{f^b:f\in\sA\}$. It is clear that the function group $\mathcal{A}^b$ separates the points in $\beta_{\sA}X$.
Since $\varepsilon:X\rightarrow\varepsilon(X)$ is a homeomorphism onto its image, there is no loss of generality in assuming that
$X=\varepsilon(X)$ and $\tau_{\sA}$ is the topology that $X$ inherits from $\prod\limits_{f\in\sA}\overline{f(X)}$.
\edfn
\medskip

\brem
The definition of the space  $\beta_{\sA}X$ above is a generalization of the well known Stone-\v{C}ech compactification $\beta X$
associated to every Tychonoff space $X$. In fact, the properties and proofs presented here are inspired in those of $\beta X$
(cf. \cite{gill-jer}).
\erem

\brem
In the sequel, if $\sA$ is a bounded function group in $G^X$ and $D\in\mathcal{D}(\sA)\cup\sE(\sA)$,
the symbol $D^b$ will denote the subset of $\beta_{\sA}X$ canonically associated to $D$ by
replacing its defining maps in $\sA$ by their continuous extensions that belong to $\mathcal{A}^b$.
For instance, if $D=f^{-1}(F)$, then $D^b=(f^{b})^{-1}(F)$.
Furthermore, in order to avoid the proliferation of awkward notation, the symbol $\overline A$ will always mean the
closure of $A$ in the $\beta_{\sA}X$ unless expressly stated otherwise.
\erem

\bprp\label{PROPOSICION 1}

Let $\sA$ be a bounded function group on $G^X$. The following assertions hold:

1. $\mathcal{A}^b$ is a function subgroup of $C(\beta_{\sA}X,G)$.

2. $D^b\cap X=D$ for all $D\in \mathcal{D}(\sA)\cup\mathcal{E}(\sA)$.

3. $\overline{D}=\overline{D^b\cap X}\subseteq D^b$ for all $D\in \mathcal{D}(\sA)$.

4. $\overline{D}\cap int(D^b)=int(\overline{D})$ for all $D\in \mathcal{D}(\sA)$.

5. $int(\overline{Z(f)})=int(Z(f^b))=\beta_{\sA}X\setminus \overline{coz(f)}$ for all $f\in\sA$.

6. $int(D^b)=int(\overline{D})$ for all $D\in \mathcal{D}(\sA)$.

7. If $P$ and $Q$ are disjoint nonempty closed subsets of $\beta_{\sA}X$,
then there exist $D_P,D_Q\in \mathcal{D}(\sA)$ such that $P\subseteq int(\overline{D_P})$, $Q\subseteq int(\overline{D_Q})$ and $D_P^b\cap D_Q^b=\emptyset$.
\eprp
\bpf

1. Obvious.

2. Since $(f^b)^{-1}(A)\cap X=f^{-1}(A)$ for all $A\subseteq G$ and $f\in \sA$, then  $(f_1^b)^{-1}(A_1)\cap (f_2^b)^{-1}(A_2)\cap X=f_1^{-1}(A_1)\cap (f_2)^{-1}(A_2)$ and $((f_1^b)^{-1}(A_1)\cup (f_2^b)^{-1}(A_2))\cap X=f_1^{-1}(A_1)\cup (f_2)^{-1}(A_2)$, for each $A_i\subseteq G$ and $f_i\in\sA$, $1\leq i\leq 2$.

3. Apply item 2.

4. It is clear that $int(\overline{D})\subseteq \overline{D}\subseteq D^b$, which yields $int(\overline{D})\subseteq \overline{D}\cap int(D^b)$.
As for the reverse inclusion, take $p\in\overline{D}\cap int(D^b)$, then there exists an open subset $U$ in $\beta_{\sA}X$ such that
$p\in U\subseteq int(D^b)\subseteq D^b$ but $U\subseteq\overline{U\cap X}\subseteq\overline{D^b\cap X}=\overline{D}$ by item 3. As a consequence $p\in U\subseteq int(\overline{D})$.

5. Let us see that $int({Z(f^b)})=\beta_{\sA}X\setminus \overline{coz(f)}$, $f\in\sA$. Indeed, since $coz(f^b)$ is an open subset
and $X$ is dense in $\beta_{\sA}X$, we have $\overline{coz(f^b)}=\overline{coz(f^b)\cap X}=\overline{coz(f})$ by item 2. Then  $\beta_{\sA}X\setminus \overline{coz(f)}=\beta_{\sA}X\setminus \overline{coz(f^b)}\subseteq\beta_{\sA}X\setminus coz(f^b)=Z(f^b)$, thus $\beta_{\sA}X\setminus \overline{coz(f)}\subseteq int(Z(f^b))$.
Moreover $int(Z(f^b))$ is an open subset disjoint with $coz(f^b)$ and, as a consequence, with its closure.
Then $\beta_{\sA}X\setminus \overline{coz(f)}=int(Z(f^b))$. On the other hand 
$int(Z(f^b))\subseteq\overline{int(Z(f^b))\cap X}\subseteq \overline{Z(f^b)\cap X}=\overline{Z(f)}$. 
Thus, we have proved that $int(\overline{Z(f)})=int(Z(f^b))=\beta_{\sA}X\setminus \overline{coz(f)}$.

6. Let $p\in int(D^b)$ and let $U$ be an open subset in $\beta_{\sA}X$ such that $p\in U$. 
Since $X$ is dense in $\beta_{\sA}X$, we have $\emptyset\neq U\cap X\cap int(D^b)\subseteq U\cap (D^b\cap X)=U\cap D$ and $p\in\overline{D}$.

7. Take $p\in P$ and $q\in Q$, respectively. Since $\mathcal{A}^b$ separates the points in $\beta_{\sA}X$,
there exists $f_{pq}\in \sA$ such that $f_{pq}^b(p)\neq f_{pq}^b(q)$. Since $G$ is metric there are two open neighborhoods $U_{pq}$ and $V_{pq}$ of $f_{pq}^b(p)$ and $f_{pq}^b(q)$ respectively in $G$ such that $\overline{U_{pq}}\cap\overline{V_{pq}}=\emptyset$ (closures in $G$). Then $p\in (f_{pq}^b)^{-1}(U_{pq})$ and $P\subseteq \bigcup\limits _{p\in P}(f_{pq}^b)^{-1}(U_{pq})$.
 $P$ being compact, there is a natural number $n_q$ such that $P\subseteq \bigcup\limits _{i=1}^{n_q}(f_{p_iq}^b)^{-1}(U_{p_iq})$ for each $q\in Q$. Then $q\in\bigcap\limits _{i=1}^{n_q}(f_{p_iq}^b)^{-1}(V_{p_iq})$ and $Q\subseteq\bigcup\limits_{q\in Q}\bigcap\limits _{i=1}^{n_q}(f_{p_iq}^b)^{-1}(V_{p_iq})$. $Q$ being compact,
 there is a natural number $m$ such that $Q\subseteq\bigcup\limits_{j=1}^m\bigcap\limits _{i=1}^{n_{q_j}}(f_{p_iq_j}^b)^{-1}(V_{p_iq_j}):=E_Q^b\subseteq\bigcup\limits_{j=1}^m\bigcap\limits _{i=1}^{n_{q_j}}(f_{p_iq_j}^b)^{-1}(\overline{V_{p_iq_j}}):=D_Q^b$.
 Moreover,
 $P\subseteq\bigcap\limits_{j=1}^m\bigcup\limits _{i=1}^{n_{q_j}}(f_{p_iq_j}^b)^{-1}(U_{p_iq_j}):=E_P^b\subseteq\bigcap\limits_{j=1}^m\bigcup\limits _{i=1}^{n_{q_j}}(f_{p_iq_j}^b)^{-1}(\overline{U_{p_iq_j}}):=D_P^b$.
 Since $\overline{U_{p_iq_j}}\cap \overline{V_{p_iq_j}}=\emptyset$, $1\leq i\leq n_{q_j}$ and $1\leq j\leq m$, we have $D_P^b\cap D_Q^b=\emptyset$.
  Therefore $Q\subseteq E_Q^b\subseteq D_Q^b$ and $E_Q^b$ is open in $\beta_{\sA}X$, that is $Q\subseteq int(D_Q^b)$ and, by item 6., $Q\subseteq int(\overline{D_Q})$.
  The same argument proves that $P\subseteq int(\overline{D_P})$.
\epf

\bcor\label{CORO1}
Let $P$ and $U$ be a closed and an open subsets of $\beta_{\sA}X$ respectively, such that $P\subseteq U$. Then there are two decreasing sequences $\{E_j\}_{j<\omega}\subseteq \mathcal{E}(\sA)$ and $\{D_j\}_{j<\omega}\subseteq \mathcal{D}(\sA)$ such that $P\subseteq E_j^b\subseteq int(D_j^b)\subseteq D_j^b\subseteq E^b_{j-1}\subseteq U$.
\ecor
\bpf
Applying the proof of Proposition \ref{PROPOSICION 1}.7, since $P$ and $\beta_{\sA}X\setminus U$ are disjoint compact subsets,
there are $E_1\in\mathcal{E}(\sA)$ and $D_1\in\mathcal{D}(\sA)$ such that $P\subseteq E_1^b\subseteq int(D_1^b)\subseteq D_1^b\subseteq U$.
Now $P$ and $\beta_{\sA}X\setminus E_1^b$ are disjoint compact subsets again and, consequently,
there are $E_2\in\mathcal{E}(\sA)$ and $D_2\in\mathcal{D}(\sA)$ such that $P\subseteq E_2^b\subseteq int(D_2^b)\subseteq D_2^b\subseteq E_1^b$.
The proof now follows by induction.
\epf

\bdfn\label{NORMAL}
 A function group $\sA$ in $G^X$ is \emph{normal}  if for each two disjoint subsets $D_1,D_2\in\mathcal{D}(\sA)$
 there are maps   $f_{ij}\in\mathcal{A^*}$ and disjoint closed subsets $F_{ij}^{(1)}, F_{ij}^{(2)}\subseteq G$, $1\leq i\leq n_j, 1\leq j\leq m$, such that
\begin{eqnarray*}
D_1\subseteq \bigcup_{j=1}^m\bigcap_{i=1}^{n_j}f_{ij}^{-1}(F_{ij}^{(1)}) & &\\
D_2\subseteq \bigcap_{j=1}^m\bigcup_{i=1}^{n_j}f_{ij}^{-1}(F_{ij}^{(2)}). & &
\end{eqnarray*}

\edfn
\medskip

In the following proposition, the closure $\overline A$ of every subset $A\subseteq X$ is understood in $\beta_{\sA^*}X$.

\begin{prop}\label{CLAUVACIA}
Let $\sA$ is a normal function group in $G^X$ that is an extension of $\sA^*$ and let
$\{D_i\}_{i=1}^n$ be a finite subset of $\mathcal{D}(\sA)$. The following assertions hold true:
\begin{itemize}
\item[(a)]  $\bigcap\limits_{i=1}^n D_i=\emptyset$\ implies\ $\bigcap\limits_{i=1}^n\overline{D_i}=\emptyset$.\\

\item[(b)] $\overline{\bigcap\limits_{i=1}^n D_i}=\bigcap\limits_{i=1}^n\overline{D_i}$.
\end{itemize}
\end{prop}

\begin{proof}
Since $\sD$ is closed under finite intersections, there is no loss of generality in assuming that $n=2$.

\noindent (a) Reasoning by contradiction, suppose there exists $p\in\overline{D_1}\cap \overline{D_2}\not=\emptyset$.
Since $\sA$ is normal, we have $D_1\subseteq \bigcup_{j=1}^m\bigcap_{i=1}^{n_j}f_{ij}^{-1}(F_{ij}^{(1)})$,
$D_2\subseteq \bigcap_{j=1}^m\bigcup_{i=1}^{n_j}f_{ij}^{-1}(F_{ij}^{(2)})$ and $F_{ij}^{(1)}\cap F_{ij}^{(2)}=\emptyset$,
for some maps $f_{ij}\in\sA^*$.
The subsets $(D_{ij}^{(1)})^b=(f_{ij}^b)^{-1}(F_{ij}^{(1)})$ and $(D_{ij}^{(2)})^b=(f_{ij}^b)^{-1}(F_{ij}^{(2)})$ are clearly disjoint in $\beta_\sA X$.
Moreover, since $p\in\overline{D_1}\subseteq D_1^{b}$, there is $j_0$ such that $p\in\bigcap_{i=1}^{n_{j_0}}(D_{ij_0}^{(1)})^b$,
in like manner $p\in\overline{D_2}\subseteq D_2^{b}$, which implies that $p\in\bigcup_{i=1}^{n_{j_0}}(D_{ij_0}^{(2)})^b$.
This is a contradiction because $(D_{ij_0}^{(1)})^b\cap(D_{ij_0}^{(2)})^b=\emptyset$.\mkp

\noindent (b) It will suffice to prove that $\overline{D_1}\cap \overline{D_2}\subseteq\overline{D_1\cap D_2}$.
 Assuming the opposite, suppose there exists
 $q\in\overline{D_1}\cap \overline{D_2}\setminus\overline{D_1\cap D_2}\subseteq \beta_{\mathcal{A}}X\setminus \overline{D_1\cap D_2}$,
 where $\beta_{\mathcal{A}}\setminus \overline{D_1\cap D_2}$ is open in $\beta_{\mathcal{A}}X$.
If $D_1\cap D_2=\emptyset$, then $\overline{D_1}\cap \overline{D_2}=\emptyset$ by (a), which is a contradiction.

So, we may assume without loss of generality that $D_1\cap D_2\neq\emptyset$. Clearly, we have that $p\neq q$ for all $p\in\overline{D_1\cap D_2}$.
 Therefore, for every $p\in\overline{D_1\cap D_2}$, there exists $f_{pq}\in\sA^*$ such that $f_{pq}^b(p)\neq f_{pq}^b(q)$.
 Since $\overline{D_1\cap D_2}$ is compact and $G$ es metric, there are
 $p_1,\dots,p_n\in\overline{D_1\cap D_2}$ and $V_1,\ldots,V_n$ neighborhoods of $f_{p_i q}(q)$ in $G$ such that
 $q\in\bigcap\limits_{i=1}^n(f_{p_i q}^b)^{-1}(V_i):=E_{q}^{b}\subseteq\bigcap\limits_{i=1}^{n}(f_{p_i q}^b)^{-1}(\overline{V_i}):=D_q^b$ and
 $D_q^b\cap\overline{D_1\cap D_2}=\emptyset$. Hence $D_q\cap D_1\cap D_2=\emptyset$.

 On the other hand $q\in \overline{D_q}$ and $D_q\in \sD(\sA)$. Thus $q\in\overline{D_q}\cap \overline{D_1}\cap\overline{D_2}$, which implies
 $D_q\cap D_1\cap D_2\not=\emptyset$. This is contradiction that completes the proof.
 \end{proof}

\bdfn\label{controlable}

 A function group $\mathcal{A}$ in $G^X$ is \textit{controllable} if for every $f\in\mathcal{A}$ and
 $D_1, D_2\in \mathcal{D}(\sA)$ with $D_1\cap D_2=\emptyset$, there is $f'\in\mathcal{A}$ and $E\in \mathcal{E}(\sA)$ such that
 $D_1\subseteq E \subseteq X\setminus D_2$, $f_{|_{D_1}}=f'_{|_{D_1}}$, and $f'_{|_{Z(f)\cup(X\setminus E)}}=e_G$.

\brem \label{control}
From the definition above, it follows that if $f''=f'f^{-1}\in \sA$ then $f''_{|_{D_1\cup Z(f)}}=e_G$ and
$f''_{|_{(X\setminus E)}}=f^{-1}_{|_{(X\setminus E)}}$.
\erem

\edfn
\bdfn\label{support}
Let $\sA$ be a bounded function group in $G^X$ and let $\phi:\sA\rightarrow G$ be a group homomorphism.  
 A closed subset $S\subseteq\beta_{\sA}X$ is a \emph{weak support} for $\phi$ if for every $f\in\sA$ such that $S\subseteq int(\overline{Z(f)})$,
 it holds $\phi(f)=e_G$. $S$ is a \emph{support} for $\phi$ if for every $f\in\sA$ such that $S\subseteq Z(f^b)$, it holds  $\phi(f)=e_G$.
\edfn
\mkp

Weak support subsets satisfy the following properties.

\bprp\label{PROPOSICION 2}

Let $\sA$ be a bounded function group in $G^X$ and let $\phi:\sA\rightarrow G$
 be a non-null  group homomorphism. The following assertions hold:

1. $\beta_{\sA}X$ is a (weak) support for $\phi$.

2. The empty set is not a (weak) support for $\phi$.

3. If $S$ is a (weak) support for $\phi$ and $S\subseteq R$ then $R$ is a (weak) support for $\phi$.

4. Let $S$ be a weak support for $\phi$, $A\subseteq X$ and $f_1,f_2\in\sA$ such that $S\subseteq int(\overline{A})$ and ${f_1}_{|_A}={f_2}_{|_A}$. Then $\phi(f_1)=\phi(f_2)$.

If in addition $\sA$ is controllable we have:

5. If $S$ and $R$ are weak supports for $\phi$ then $S\cap R\neq\emptyset$.
\eprp
\bpf
1. If $\beta_{\sA}X\subseteq int(\overline{Z(f)})$, $f\in\beta_{\sA}X$, then $X=\beta_{\sA}X\cap X\subseteq int(\overline{Z(f)})\cap X\subseteq Z(f^b)\cap X=Z(f)$ and $f= e_G$. Then $\phi(f)=\phi(e_G)=e_G$.

2. and 3. are obvious.

4. Set $f=f_1f_2^{-1}$. Then $S\subseteq int(\overline{A})\subseteq int(\overline{Z(f)})$, which yields $\phi(f_1f_2^{-1})=\phi(f)=e_G$. Therefore $\phi(f_1)=\phi(f_2)$.

5. Let $S$ and $R$ be weak supports for $\phi$ and suppose that $S\cap R=\emptyset$.
By Proposition \ref{PROPOSICION 1}.7, there are two subsets $D_S$ and $D_R$ in $\mathcal{D}(\sA)$ such that $S\subseteq int(\overline{D_S})$, $R\subseteq int(\overline{D_R})$ and $D_S^b\cap D_R^b=\emptyset$. Take $f\in\sA$ such that $\phi(f)\neq e_G$ and apply that $\sA$ is controllable to obtain $E\in \mathcal{E}(\sA)$ and $f'\in\sA$ such that $D_S\subseteq E\subseteq X\setminus D_R$, $f'_{|_{D_S}}=f_{|_{D_S}}$ and $Z(f)\cup(X\setminus E)\subseteq Z(f')$. Then by item 4. $\phi(f')=\phi(f)\neq e_G$ and, since $R\subseteq int(\overline{D_R})\subseteq int(\overline{Z(f')})$ then $\phi(f')=e_G$, which is a contradiction.
\epf

\bdfn
Let $\sA$ be a function group in $G^X$. Two maps $f,g\in \sA$ are \textit{separated} (resp. \emph{detached}) if $coz(f)\cap coz(g)=\emptyset$
(resp. if there are $D_1,D_2\in \mathcal{D}(\sA)$ such that $coz(f)\subseteq D_1$,
$coz(g)\subseteq D_2$ and $D_1\cap D_2=\emptyset$).\mkp

\noindent A group homomorphism $\phi:\sA\rightarrow G$ is \textit{separating} (resp. \textit{weakly separating})  if
 for every separated (resp. detached) maps $f,g\in\sA$, it holds that either $\phi(f)=e_G$ or $\phi(g)=e_G$.
\edfn
Notice that every separating homormorphism $\phi$ is weakly separating.

\bprp\label{detaching}
  Let $\sA$ be a function group in $G^X$ and let $\phi\colon\sA\to G$ be a weakly separating homomorphism.
Then for every $f_1, f_2$ in $\sA$ such that $\overline{coz(f_1)}\cap \overline{coz(f_2)}=\emptyset$,
either $\phi(f_1)=e_G$ or $\phi(f_2)=e_G$ (here, the closures are taken in $\beta_{\sA}X$).
If in addition $\sA$ is normal, then the converse implication is also true.
\eprp

\bpf

($\Rightarrow$)  (Notice that the normality of $\sA$ is not needed in this implication). Let $f_1,f_2\in\sA$
such that $\overline{coz(f_1)}\cap\overline{coz(f_2)}=\emptyset$. By Proposition \ref{PROPOSICION 1}.7,
there is $D_i\in\sD(\sA^*)$ such that $\overline{coz(f_i)}\subseteq int(\overline{D_i})\subseteq {D}_i^b$, $1\leq i\leq 2$,
and ${D}_1^b\cap {D}_2^b=\emptyset$. Since $\phi$ is weakly separating and $coz(f_i)\subseteq {D}_i^b\cap X=D_i$, $1\leq i\leq 2$,
it follows that either $\phi(f_1)=e_G$ or $\phi(f_2)=e_G$.

($\Leftarrow$) Let $f_1,f_2\in\sA$ be two detached maps in $\sA$.  Then there are $D_i\in\sD(\sA)$ such that $coz(f_i)\subseteq D_i$, $1\leq i\leq 2$,
and $D_1\cap D_2=\emptyset$. By Proposition \ref{CLAUVACIA}, we have $\overline{coz(f_1)}\cap\overline{coz(f_2)}\subseteq \overline{D_1\cap D_2}=\emptyset$.
Therefore, either $\phi(f_1)=e_G$ or $\phi(f_2)=e_G$.
\epf
\medskip

Now, we will prove that every non-null weakly separating group homomorphism
$\phi:\sA\rightarrow G$, where $\sA$ is a controllable, bounded, function group, has the smallest possible
weak support set. For that purpose set
$\mathcal{S}= \{S\subseteq\beta_{\sA}X: S \text{ is a weak support for }\phi\}$.

There is a canonical partial order that can be defined on $\mathcal{S}$:
$S\leq R$ if and only if $R\subseteq S$. A standard argument shows that $(\mathcal{S},\leq)$ is an inductive set and, by Zorn's lemma,
$\mathcal{S}$ has a $\subseteq$-minimal element $S$. Furthermore, this minimal
element is in fact a minimum because of the next proposition.

\bprp\label{PROPOSICION_3}

 Let $\sA$ be a bounded function group in $G^X$ and let $\phi:\sA\rightarrow G$ be a non-null weakly separating group homomorphism.
 If $\sA$ is controllable then the minimum
element of $\mathcal{S}$ is a singleton.
\eprp
\bpf
Let $S$ a minimal element of $\mathcal{S}$, which is nonempty by Proposition \ref{PROPOSICION 2}.2.
Suppose now that there are two different points $p_1,p_2$ in $S$. By Proposition \ref{PROPOSICION 1}.7,
we can select two subsets $D_1,D_2\in \mathcal{D}(\sA)$ such that $p_1\in int(\overline{D_1})$, $p_2\in int(\overline{D_2})$ and $D_1^b\cap D_2^b=\emptyset$.
Since $S$ is minimal, the subset $S\setminus int(\overline{D_i})$ is a compact subset that is not a weak support for $\phi$, $1\leq i\leq 2$.
Hence, there is $f_i\in\sA$ such that $S\setminus int(\overline{D_i})\subseteq int(\overline{Z(f_i)})$ and $\phi(f_i)\neq e_G$, $1\leq i\leq 2$.
Then $S\subseteq int(\overline{Z(f_1)})\cup int(\overline{Z(f_2)}):=U$ and, by Proposition \ref{detaching}, we have  $\emptyset\neq \overline{coz(f_1)}\cap \overline{coz(f_2)}=:C$.
Since, by Proposition \ref{PROPOSICION 1}.5, $int(\overline{Z(f_i)})=\beta_{\sA}X\setminus \overline{coz(f_i)}$,  $1\leq i\leq 2$,
we have $S\cap C\subseteq U\cap C=\emptyset$.  Applying Proposition \ref{PROPOSICION 1}.7 again, there are $D_S,D_C\in \mathcal{D}(\sA)$
such that $S\subseteq int(\overline{D_S})$, $C\subseteq int(\overline{D_C})$ and $D_S^b\cap D_C^b
=\emptyset$.

We now apply that $\sA$ is controllable to $D_S$, $D_C$ and $f_1$ in order to obtain a subset $E\in \mathcal{E}(\sA)$ and a map $f'_1\in \sA$ such that
$D_S\subseteq E\subseteq X\setminus D_C$, ${f'_{1}}_{|{D_S}}={f_1}_{|{D_S}}$ and ${f'_1}_{|{Z(f_1)\cup(X\setminus E)}}=e_G$.
By Proposition \ref{PROPOSICION 2}.4, we have $\phi(f'_1)=\phi(f_1)\neq e_G$ and, by Proposition \ref{detaching}, $\overline{coz(f'_1)}\cap \overline{coz(f_2)}\neq\emptyset$.
On the other  hand, $Z(f_1)\subseteq Z(f'_1)$ and $C\subseteq int(\overline{D_C})\subseteq int(\overline{Z(f'_1}))$.
We have got a contradiction because $\emptyset\neq \overline{coz(f'_1)}\cap \overline{coz(f_2)}\subseteq \overline{coz(f_1)}\cap \overline{coz(f_2)}= C\subseteq int(\overline{Z(f'_1)})$,
which is impossible by Proposition \ref{PROPOSICION 1}.5.

We have proved that $|S|=1$. Moreover, if $R$ is a weak support for $\phi$, then by Proposition \ref{PROPOSICION 2}.5,
$S\subseteq R$. Thus the set $S$ is the minimum element in $\mathcal{S}$ and this complete the proof.
\epf
\bigskip

\section{Extension of bounded function groups}

In this section, we are concerned with functions groups that are extensions of  their subgroups of bounded functions.
Therefore, every function group $\sA$ is assumed to be a $\kappa$-extension of its bounded subgroup $\sA^*$
for some undetermined cardinal number $\kappa$ from here on.

 Observe that given a bounded function group $\sA^*$, there is a \emph{largest $\kappa$-extension} canonically associated to $\sA^*$ for every cardinal number $\kappa$; namely
 $$ext_\kappa(\sA^*)\defi \{f\in G^X : f=\prod\limits_{i<\kappa}f_i : f_i\in\sA^*,\ \text{with}\ \{coz(f_i):i<\kappa\}\ \text{locally finite}\}.$$
\medskip

In the sequel, we deal with the Hausdorff compactification $\beta_{\sA^*}X$ associated
to the bounded function subgroup $\sA^*\subseteq\sA$ on a set $X$. From here on, the symbol $\overline A$ will mean the closure of $A$ in the $\beta_{\sA^*}X$
unless expressly stated otherwise. The next result will be used several times along the paper. We omit its easy verification.

 \bprp\label{complemento}
Let $\sA$ be a function group in $G^X$ such that  $\tau_\sA=\tau_{\sA^*}$. Then $int(\overline{Z(f)}) =\beta_{\sA^*}X\setminus\overline{coz(f)}$ for all $f\in\sA$.
\eprp



The definition of \emph{support set} (resp.  \emph{weak support set}) for a general function group $\sA$ is analogous to
the definition given for bounded functions groups in  Definition \ref{support}. So, given a group homomorphism
$\phi\colon \sA\to G$, we say that
a closed subset $S\subseteq\beta_{\sA^*}X$ is a \emph{weak support} for $\phi$ if for every $f\in\sA$ such that $S\subseteq int(\overline{Z(f)})$,
it holds $\phi(f)=e_G$. The set $S\subseteq X$ is a \emph{support} for $\phi$ if for every $f\in\sA$ such that $S\subseteq Z(f)$, it holds  $\phi(f)=e_G$.

  \blem\label{extensionsoporte} Let $\sA$ be a function group in $G^X$ that is an extension of its bounded subgroup $\sA^*$, and
let $\phi\colon \sA\to G$ be a weakly separating group homomorphism such that $\phi|_{\sA^*}$ is non-null and $\sA^*$ is controllable.
Then the following assertions hold:
\begin{itemize}
\item[(a)] If $S\subseteq \beta_{\sA^*}X$ is a weak support for $\phi|_{\sA^*}$ then $S$ is a weak  support for $\phi$.
\item[(b)] If a $\tau_{\sA}$-compact subset $S\subseteq X$ is a support for for $\phi|_{\sA^*}$ then $S$ is a  support for $\phi$.
\end{itemize}
\elem
\begin{proof}

(a) Suppose that $S$ is a support for $\phi|_{\sA^*}$ but not a support for $\phi$. Then there is $f\in\sA$ such that $S\subseteq int(\overline{Z(f)})$ and $\phi(f)\neq e_G$.
By Corollary \ref{CORO1} and Proposition 3.2.6,  we can take $D_1,D_2\in\sD(\sA)$ and $E_1\in\sE(\sA)$ such that
$S\subseteq int(\overline{D_2})\subseteq D_2^b\subseteq E_1^b\subseteq D_1^b\subseteq int(\overline{Z(f)})$.

On the other hand, since $\phi|_{\sA^*}$ is non-null, there is $f_1\in \sA^*$ such that $\phi(f_1)\neq e_G$.
Applying that $\sA^*$ is controllable to $f_1$, $D_2$ and $D^{(2)}=X\setminus E_1$, we obtain
${f_1'}\in\sA^*$ and  $E\in\sE(\sA)$ such that $D_2\subseteq E\subseteq E_1$, $f'_1|_{D_2}=f_1|_{D_2}$ and ${f_1'}|_{Z(f_1)\cup(X\setminus E)}=e_G$.
Then, by Proposition \ref{PROPOSICION 2}.4, $\phi({f'}_1)=\phi(f_1)\neq e_G$ and $coz(f'_1)\subseteq E\subseteq D_1$.
Hence $\overline{coz(f'_1)}\subseteq D_1^b$ and
$\overline{coz(f'_1)}\cap \overline{coz(f)}\subseteq D_1^b\cap (\beta_{\sA^*}X\setminus int(\overline{Z(f)}))=\emptyset$,
which is a contradiction by Proposition \ref{detaching}.

(b) Assume that the $\tau_{\sA}$-compact subset $S\subseteq X$ is a support set for $\phi|_{\sA^*}$ and take $f\in \sA$ such that $f|_{S}=e_G$.
Since $\sA$ is an extension of $\sA^*$,
we have $f=\prod\limits_{i\in I} f_i$ with $f_i\in\sA^*$ for all $i\in I$ and the family $\{coz(f_i):i\in I\}$ being locally finite in $X$.
Now, since $S$ is compact, there is an open subset
$U\in \tau_\sA$ that contains $S$ and intersects finitely many members of $\{coz(f_i):i\in I\}$. That is, there is a finite subset $J\subseteq I$ such that
$U\cap coz(f_i)=\emptyset$ for each $i\in I\setminus J$. Set $f_J\defi \prod \{f_i : i\in J\}$
and $f^J\defi \prod \{f_i : i\in I\setminus J\}$.
We have $f=f_J\cdot f^J$, which yields $\phi(f)= \phi(f_J)\cdot \phi(f^J)$. Now, remark that $S\subseteq int(\overline{Z(f^J)})$,
which implies that $\phi(f^J)=e_G$ by assertion (a). On the other hand, $f_J\in\sA^*$ and $f_J|_S=f|_S=e_G$,
which yields $\phi(f_J)=e_G$. Thus $\phi(f)=e_G$, which implies that $S$ is a support set for $\phi$.
\end{proof}

\bprp\label{LEMA 1}
Let $\sA$ be a function group in $G^X$ that is an extension of its bounded subgroup , and let
$\phi:\mathcal{A}\rightarrow G$ be a weakly separating group homomorphism. If $\phi|_{\sA^*}$ is non-null and
continuous with respect to the pointwise convergence topology, and $\sA^*$ is controllable,
then the minimum weak support for $\phi|_{\sA^*}$ is a support for $\phi$ and is placed in $X$.
\eprp
\bpf
By Lemma \ref{extensionsoporte}, it will suffice to prove that the minimum weak support for
$\phi|_{\sA^*}$ is a support set for $\phi|_{\sA^*}$ and that is placed in $X$.
Let $\{p\}$ be the minimum weak support for $\phi|_{\sA^*}$. Reasoning by contradiction, suppose $p\notin X$
and let $F\in Fin(X)$, where $Fin(X)$ denotes the set of all finite subsets of $X$, (partially) ordered by inclusion.
By Proposition \ref{PROPOSICION 1}.7 there are $D_p,D_F\in\mathcal{D}(\sA)$ such that $p\in int(\overline{D_p})$, $F\in int(\overline{D_F})$ and $D_p^b\cap D_F^b=\emptyset$.
Since $\phi|_{\sA^*}$ is non-null, there exists $f_0\in\mathcal{A^*}$ such that $\phi(f_0)\neq e_G$. Applying the controllability of $\mathcal{A^*}$,
we can take $f_F\in\mathcal{A^*}$ and $E\in\mathcal{E}(\sA)$ such that $D_F\subseteq E\subseteq X\setminus D_p$, ${f_F}_{|{D_F}}={f_0}_{|{D_F}}$
and ${f_F}_{|{Z(f_0)\cup(X\setminus E)}}=e_G$. Then ${f_F}_{|{D_p}}=e_G$, which implies, by Proposition \ref{PROPOSICION 2}.4,
that $\phi(f_F)=e_G$. 
Therefore the net $(f_F)_{F\in Fin(X)}$ converges pointwise to $f_0$ in $\mathcal{A^*}$.
As a consequence $(\phi(f_F))_{F\in Fin(X)}=(e_G)$ converges to $\phi(f_0)\neq e_G$, which is a contradiction. Thus $p\in X$.

Let us now verify that $\{p\}$ is a support for $\phi|_{\sA^*}$. Take $f\in\mathcal{A^*}$ such that $p\in Z(f^b)$. Then $p\in Z(f^b)\cap X=Z(f)$.

We have $\{p\}=\bigcap\{\overline{D}:p\in int(\overline{D}),D\in\mathcal{D}(\sA)\}$. Indeed, if $q\neq p$, $q\in\beta_{\mathcal{A}^*}X$,
then by Proposition \ref{PROPOSICION 1}.7, there are $D_q,D_p\in\mathcal{D}(\sA)$ such that $q\in int(\overline{D_q})$, $p\in int(\overline{D_p})$ and $D_q^b\cap D_p^b=\emptyset$.
Hence $q\in\beta_{\mathcal{A}^*}X\setminus \overline{D_p}$ and $q\in\bigcup\{\beta_{\mathcal{A}^*}X\setminus\overline{D}:p\in int(\overline{D}),D\in\mathcal{D}(\sA)\}$.
That is, the set $\mathcal{D}_p=\{D\in\mathcal{D}(\sA):p\in int(\overline{D})\}$ is a neighborhood base at $p$. 
Let $D\in\sD_p$, then by Corollary \ref{CORO1} there is $E_1(D)\in\sE(\sA)$ and $D_2(D),D_1(D)\in\sD(\sA)$ such that $p\in D_2(D)\subseteq E_1(D)\subseteq D_1(D)\subseteq int(\overline{D})$. We now apply Remark \ref{control} to $f^{-1}$, $D_2(D)$ and $D_o(D)=X\setminus E_1(D)$ to obtain $f_D\in\sA^*$ and  $E(D)\in\sE(\sA)$ such that $D_2(D)\subseteq E(D)\subseteq E_1(D)$, ${f_D}_{|D_2(D)\cup Z(f)}=e_G$ and ${f_D}_{|X\setminus E(D)}=f_{|X\setminus E(D)}$, which implies ${f_D}_{|(X\setminus D)}=f_{|(X\setminus D)}$.

We claim that $(f_D)_{D\in\mathcal{D}_p}$ converges pointwise to $f$. Indeed, let $F\in Fin(X)$. If $p\notin F$, by Proposition \ref{PROPOSICION 1},
there is $D_0\in \sD(\sA)$ such that $p\in int(\overline{D_0})$ and $D_0\cap F=\emptyset$, then $F\subseteq X\setminus D_0\subseteq X\setminus D$,
for all $D\subseteq D_0$, $D\in\sD_p$, and ${f_D}_{|F}=f_{|F}$. If $F=\{p\}\cup F_1$ then $f(p)=e_G=f_D(p)$ and, reasoning as above, we can take $D_1\in\sD_p$ such that ${f_D}_{|F_1}=f_{|F_1}$ for all $D\subseteq D_1$, $D\in\sD_p$. Since $\phi|_{\sA^*}$ is continuous with respect to the pointwise convergence topology, we have that $(\phi(f_D))_{D\in\mathcal{D}_p}=(e_G)$ converges to $\phi(f)$.
That is $\phi(f)=e_G$, which implies that $p$ is a support set for $\phi$.
\epf

\bdfn
Let $\mathcal{A}$ be a subgroup of $G^X$. A group homomorphism $\phi:\mathcal{A}\rightarrow G$ is non-vanishing if for every $f,g$ in $\sA$ such that $Z(f)\cap Z(g)=\emptyset$,
it holds that $\phi(f)\neq e_G$ or $\phi(g)\neq e_G$.
\mkp
\edfn

\brem\label{NOTA_1}
Every element $a\in G$ can be identified with the constat map $``a$'' belonging to $G^X$. In this sense, if $\sA$ is a subgroup of $G^X$, we say that $\sA$ \emph{contains the constants maps}
if $G\subseteq \sA$. We write $Epi(G)$ to denote the set of all homomorphisms from $G$ onto $G$. We will also denote by $d$ to the metric defined on $G$
and $B(a,r)$ will mean the open ball centered in $a$ with radius $r$.
\erem

\bprp\label{LEMA 2}
Let $\mathcal{A}$ be a function group in $G^X$ that is an extension of its bounded subgroup $\sA^*$,
which is controllable.
Suppose that $G\subseteq \mathcal{A}$ and $\phi_{|G}\subseteq Epi(G)$.
If $\phi:\mathcal{A}\rightarrow G$ is a separating non vanishing group homomorphism then the minimum weak support $\{p\}$
for $\phi|_{\sA^*}$ is a support set for $\phi|_{\sA^*}$.
\eprp
\bpf
Reasoning by contradiction, suppose there is $f\in\mathcal{A^*}$ such that $f^b(p)=e_G$ but $\phi(f)\neq e_G$. Since $\phi_{|G}\subseteq Epi(G)$,
there is $a\in G\setminus\{e_G\}$ such that $\phi(f)=\phi(a)\neq e_G$. Set $r=d(a,e_G)$ and $U=(f^b)^{-1}(B(e_G,r/2))$, then $f^b(q)\neq a$ for all $q\in U$. As a consequence $(a^{-1}f^b)(q)\neq e_G$ for all $q\in U$. By Proposition \ref{PROPOSICION 1}, we can take $D_p$ and $D_U$ in $\mathcal{D}(\sA)$ such that $p\in int(\overline{D_p})$,
$\beta_{\mathcal{A}^*}X\setminus U\subseteq int(\overline{D_U})$ and $D_p^b\cap D_U^b=\emptyset$. Applying Remark \ref{control},
there is $E\in \mathcal{E}(\sA)$ and $g\in\mathcal{A^*}$ such that $g_{|D_p}=e_G$ and $g_{|(X\setminus E)}=a$, that is $\phi(g)=e_G$ and $g^b_{|int(\overline{D_U})}=a\neq e_G$.
We have $Z(g^b)\subseteq U$ and $Z(a^{-1}f^b)\subseteq \beta_{\mathcal{A}^*}X\setminus U$. Hence $Z(g)\cap Z(a^{-1}f)=\emptyset$ but $\phi(g)=e_G$ and $\phi(a^{-1}f)=\phi(a)^{-1}\phi(f)=e_G$,
which is a contradiction.
\epf

\bprp\label{LEMA 3}
Let $\mathcal{A}$ be a function group in $G^X$ that is an extension of its bounded subgroup $\sA^*$,
which is controllable.
If $\phi:\mathcal{A}\rightarrow G$ is a separating  group homomorphism
such that $\phi|_{\sA^*}$ is non-null and $G$ is a discrete group,
then the minimum weak support $\{p\}$ for $\phi|_{\sA^*}$ is a support set for $\phi|_{\sA^*}$.
\eprp
\bpf
If $G$ is discrete $Z(f^b)=(f^b)^{-1}(e_G)$ is clopen in $\beta_{\mathcal{A}^*}X$.
Thus, if $f\in\sA^*$, it follows that $f(X)$ is a finite subset of $G$. As a consequence,
it is easily verified that $Z(f^b)=\overline{Z(f)}=int(Z(f^b))=int(\overline{Z(f)})$.
Hence, if $p$ is a weak support, then it is automatically a support set.
\epf

\bprp\label{PROPOSICION_4}
Let $\mathcal{A}$ be a function group in $G^X$ that is an extension of its bounded subgroup $\sA^*$,
which is controllable.
Let $\phi:\mathcal{A}\rightarrow G$ be a separating group homomorphism
such that $\phi|_{\sA^*}$ is non-null, $G\subseteq\mathcal{A}$ and $\phi_{|G}$ is continuous with respect to the topology of $G$.
If $\{p\}$ is a support for $\phi|_{\sA^*}$, then the following assertion hold:
\begin{itemize}
\item[(a)] $\phi|_{\sA^*}$ is continuous with respect to the uniform convergence topology.

\item[(b)] If $p\in X$ then $\phi$ is continuous with respect to the pointwise convergence topology.
\end{itemize}
\eprp
\bpf
(a) Let $(f_i)_{i\in I}$ be a net converging uniformly to $f$ in $\mathcal{A^*}$, then $d(f_i(x),f(x))$ converges uniformly to $0$ for all $x\in X$.
As a consequence $d(f_i^b(q),f^b(q))$ converges uniformly to $0$ for all $q\in\beta_\mathcal{A^*}X$.
In particular $(f_i^b)_{i\in I}$ converges to $f^b$ in the uniform convergence topology. Take the constant maps $a_i\defi f^b_i(p)$ and $a\defi f^b(p)$.
Then $(\phi(a_i))_{i\in I}$ converges to $\phi(a)$ by hypothesis. Since $(a_i^{-1}f_i)^b(p)=a_i^{-1}f_i^b(p)=e_G$ and $(a^{-1}f)^b(p)=a^{-1}f^b(p)=e_G$,
Proposition \ref{LEMA 2} yields that $\phi(a_i^{-1}f_i)=e_G$ and $\phi(a^{-1}f)=e_G$.
Thus $\phi(f_i)=\phi(a_i)$ and $\phi(f)=\phi(a)$, which yields $(\phi(f_i))_{i\in I}$ converges to $\phi(f)$.

(b) The same arguments as in item (a) can be applied here using Lemma \ref{extensionsoporte}\,(b).
\epf

\section {Separating group homomorphisms}

From here on, $\sA$ and $\sB$ will denote two function groups in $G^X$ and $G^Y$
that are extensions of their bounded subgroups $\sA^*$ and $\sB^*$, respectively.
Denote by $\delta_x:\sA\to G$ the evaluation map at the point $x$, that is $\delta_x(f)=f(x)$ for every $f\in\sA$.
It is said that $\sA$ is \emph{pointwise dense} when $\delta_x(\sA)$ is dense in $G$ for all $x\in X$.

\bdfn
A group homomorphism $H\colon\sA\to\mathcal{B}$ is named  \textit{separating} (resp. \textit{weakly separating})
if given two separated (resp. detached) maps $f$ and $g$ in $\sA$,
it holds that $Hf$ and $Hg$ are separated (resp. detached) in $\sB$.

A group homomorphism $H\colon\sA\to\mathcal{B}$ is named \textit{(weakly) biseparating} if it is bijective and both $H$ and $H^{-1}$ are (weakly) separating.
\edfn
\mkp

Remark that we may assume that $H(\sA^*)$ is faithful on $Y$.
Otherwise, we would replace $Y$ by $\bigcup\{cozH(f) : f\in\sA^*\}$.\mkp


\bdfn
If $H\colon \sA\to\mathcal{B}$ is a weakly separating group homomorphism and assume that $\sA^*$ is controllable.
The map $\delta_y\circ H$ is a weakly separating group homomorphism of $\sA$ into $G$ for all $y\in Y$.
Furthermore, since $\sA^*$ is controllable,
the set $$\mathcal{S}_y=\{S\subseteq \mathcal{K}_{\sA^*}: S\text{ is a weak support for }\delta_y\circ H \}$$
contains a singleton as the minimum element of $\sS$, which we denote by $s_y$.
Applying the results established in the precedent section to the homomorphism $\delta_y\circ H$,
we can define the \textit{weak support map} that is canonically associated to $H$:
$$h\colon Y\to \mathcal{K}_{\sA^*} \text{ by } h(y)= s_y.$$
In case $h(y)$ is also a support set for $\delta_y\circ H$ for all $y\in Y$, then we say that $h$ is the \emph{support map} associated to $H$.
\edfn
\mkp

Let us assume, from here on, that the sets $X$ and $Y$ are equipped with the topologies canonically associated to them by $\sA$ and $\sB$ respectively.

\bprp\label{hcontinua}
Let $\sA$ and $\sB$ be two function groups in $G^X$ and $G^Y$ such that $\sA^*$ is controllable,
and let $H\colon\sA\to\mathcal{B}$ be a weakly separating homomorphism.
Then the weak support map $h\colon Y\to \beta_{\sA^*}X$ is continuous.
\eprp
\begin{proof}
Let $(y_d)_{d\in D}$ be a net in $Y$ converging to $y\in Y$. Then  $(h(y_d))_{d\in D}\subseteq\beta_{\sA^*}X$.
By a standard compactness argument, we may assume WLOG that $(h(y_d))_d$ converges to $p\in \beta_{\sA^*}X$. Reasoning by contradiction, suppose $h(y)\neq p$.
By Proposition \ref{PROPOSICION 1}.7, we take two disjoint subsets $D_{y}$ and $D_p$ in $\sD(\sA)$ such that $h(y)\in int(\overline{D_y})$,
$p\in int (\overline{D_p})$,
and $D^b_y\cap D^b_p=\emptyset$. There is an index $d_1$ such that $h(y_d)\in int(\overline{D_p})$ for all $d\geq d_1$.

 Every weak support set for $\delta_{v}\circ H$ contains $h(v)$ for all $v\in Y$.
 Hence, the nonempty compact subset $\beta_{\sA^*}X\setminus int(\overline{D_y})$ may not be a weak support for $\delta_y\circ H$.
So, there exists $f\in\sA^*$ such that $\beta_{\sA^*}X\setminus int(\overline{D_y})\subseteq int(\overline{Z(f)})$ and $(\delta_{y}\circ H)(f)=Hf(y)\neq e_G$.
Since $G$ is metrizable, there are two open neighborhoods $U_y$ and $V_{e_G}$ of $Hf(y)$ and $e_G$ respectively, such that $\overline{U_y}\cap \overline{V_{e_G}}=\emptyset$ (closures in $G$).
Moreover, since $H(f)$ is a continuous map, we have that $(Hf(y_d))_d$ converges to $Hf(y)$. Then, there is $d_2\geq d_1$ such that $Hf(y_d)\in U_y$ for all $d\geq d_2$,
which yields $Hf(y_d)\neq e_G$ for all $d\geq d_2$. On the other hand, since $d_2\geq d_1$,
the set $\beta_{\sA^*}X\setminus int(\overline{D_p})$ is not a weak support subset for $\delta_{y_{d_2}}\circ H$.
Hence, there exists $f_2\in\sA^*$ such that $\beta_{\sA^*}X\setminus int(\overline{D_p})\subseteq int(\overline{Z(f_2)})$ and $Hf_2(y)\neq e_G$.
So, we have that $y_{d_2}\in coz(Hf_2)\cap coz(Hf)$ and, since $H$ is a weakly separating map, Proposition \ref{detaching} and Proposition \ref{PROPOSICION 1} yield
$\emptyset\neq \overline{coz(f_2)}\cap \overline{coz(f})=(\beta_{\sA^*}X\setminus int(\overline{Z(f_2)}))\cap(\beta_{\sA^*}X\setminus int(\overline{Z(f)}))\subseteq int(\overline{D_p})\cap int(\overline{D_y})\subseteq D^b_p\cap D_y^b=\emptyset$. This contradiction completes the proof.
\end{proof}
\medskip

We show next that if $H(\sA^*)\subseteq \sB^*$, the (weak) support map $h$ can be extended to a continuous map on $\beta_{\sB^*}Y$.

\bprp\label{*controllable}
Let $\sA$ and $\sB$ be two function groups in $G^X$ and $G^Y$ such that $\sA^*$ is controllable. If $H\colon\sA\to\mathcal{B}$ is a weak separating
homomorphism such that $H(\sA^*)\subseteq \sB^*$, then the following assertions hold:
\begin{enumerate}[(a)]
\item  There is a map $h^b:\beta_{\sB^*}Y\to\beta_{\sA^*}X$ such that $h^b_{|Y}=h$; that is, $h^b$ extends the canonical weak support map $h$
from $Y$ to $\beta_{\sA^*}X$ associated to $H$.
\item  $h^b$ is continuous.
\item  If $H$ is injective then $h^b$ is onto.
\end{enumerate}
\eprp
\bpf
(a) Define the map $H^b:\sA^{b}\to\sB^b$ as $H^bf^b\defi (Hf)^b$ for each $f\in\sA^*$. It is readily seen that the map $H^b$ is a group homomorphism.
We claim that $H^b$ is also weakly separating. Indeed, let $f,g\in\sA^*$ such that $f^b$ and $g^b$ are detached, which means that $f$ and $g$
are detached in $X$. Since $H$ is weakly separating
we have $coz(Hf)\cap coz(Hg)=\emptyset$. Since $Y$ is dense in $\beta_{\sB^*}Y$ and, by Proposition \ref{PROPOSICION 1}.2,
$E^b\cap Y=E$ for all $E\in \sE(\sB)$, 
we obtain $coz(H^bf)\cap coz(H^bg)=\emptyset$.

Thus, the map $\delta_q\circ H^b:\sA^b\to G$, where $\delta_q$ is the evaluation map on $q\in\beta_{\sB^*}Y$,
is a non-null weakly separating homomorphism that has a minimum weak support $s_q\in\beta_{\sB^*}Y$.
Therefore, we have defined the weak support map
$$h^b:\beta_{\sB^*}Y\to\beta_{\sA^*}X\text{ by }h^b(q)=s_q.$$
It is straightforward to verify that $h^b_{|Y}=h$.

(b) The proof is analogous to Proposition \ref{hcontinua}.

(c) Suppose $h^b(\beta_{\sB^*}Y)\neq\beta_{\sA^*}X$ and pick up $p\in \beta_{\sA^*}X\setminus h^b(\beta_{\sB^*}Y)$.
Since $h^b$ is continuous and $\beta_{\sB^*}Y$ is compact, we have that $h^b(\beta_{\sB^*}Y)$ is closed subset of $\beta_{\sA^*}X$.
Hence, by Proposition \ref{PROPOSICION 1}.7, there are $D_p,D\in \sD(\sA)$ such that
$p\in int(\overline{D_p})$, $h^b(\beta_{\sB^*}Y)\in int(\overline{D})$ and $D^b_p\cap D^b=\emptyset$.
On the other hand, since $X$ is dense in $\beta_{\sA^*}X$ and $\sA^*$ is faithful,
there is $x_p\in int(\overline{D_p})\cap X$ and $f\in\sA^*$ such that $f(x_p)\neq e_G$.
Applying that $\sA^*$ is controllable we obtain $f'\in\sA^*$ and $E\in\sE(\sA)$ such that
$D_p\subseteq E\subseteq X\setminus D$, $f'_{|D_p}=f_{|D_p}$ and $f'_{|Z(f)\cup(X\setminus E)}=e_G$.
Then $f'(x_p)=f(x_p)\neq e_G$ and $h^b(\beta_{\sB^*}Y)\subseteq int(\overline{D})\subseteq int(\overline{Z(f')})$,
which implies ${Hf'}^b(q)=e_G$ for all $q\in\beta_{\sB^*}Y$.
Thus $(Hf')(y)=e_G$ for all $y\in Y$ and, since $H$  is injective, we obtain $f'=e_G$, which is a contradiction.
\epf

\section{Main results}
\bdfn\label{upsilon}
Given a topological space $(T,\tau)$, we say that $V\subseteq T$ is a \emph{$\tau(\kappa)$-set} if $V=\bigcap\limits_{i<\kappa}V_i$,
where each $V_i\in\tau$. Here $\tau(\kappa)$ denotes the topology on $T$, whose open basis consists of all
$\tau(\kappa)$-sets. Open sets in this topology are named $\tau(\kappa)$-open sets. In like manner,
closed subsets $C\subseteq T$ in this topology are named $\tau(\kappa)$-closed sets. Thus, a subset $C\subseteq T$ is $\tau(\kappa)$-closed
if for every $p\in T\setminus C$ there is a family of open subsets $\{V_i\subseteq T :i<\kappa\}$ such that
$p\in\bigcap\limits_{i<\kappa}V_i$ and $\bigcap\limits_{i<\kappa}V_i\cap C=\emptyset$. Of course, in this definition,
$\tau(\omega)$-sets coincide with the well known notion of  \emph{$G_\delta$-set}, namely countable intersection of open sets.

Now assume that $\sA$ is function group in $G^X$ that is an extension of its bounded subgroup $\sA^*$,
and let $\beta_{\sA^*}X$ be the Hausdorff compactification of $X$ associated to $\sA^*$.
We define the set $$\kappa\upsilon_{\sA^*}X\defi \cap\{C\subseteq \beta_{\sA^*}X: X\subseteq C \text{ and } C \text{ is a } \tau(\kappa)\text{-closed subset}\}.$$
Remark that for every $p\in\kappa\upsilon_{\sA^*}X$ and every $\tau(\kappa)$-set $V$ containing $p$,
we have $V\cap X\neq\emptyset$, that is, $\overline{X}^{\tau(\kappa)}=\kappa\upsilon_{\sA^*}X$.

For simplicity, when $\kappa =\omega$, we will denote the space $\omega\upsilon_{\sA^*}X$ simply as $\upsilon_{\sA^*}X$.
For the sake of simplicity, we will deal with the case $\kappa=\omega$ here, but our results are easily generalized for any cardinal number $\kappa$.
\edfn
\mkp

We next prove that when the group $G$ is discrete and  $\sA$ is a normal function group in $G^X$,
then every map in $\sA$ can be extended to a continuous map defined on $\upsilon_{\sA^*}X$.

\blem\label{extension}
Let  $\sA$ be a normal function group in $G^X$ and assume that $G$ is a discrete group.
Then for every $f\in\sA$ there is a continuous map $f^{\gu}$ of $\upsilon_{\sA^*}X$ into $G$ such that the following diagram commutes
\[
\xymatrix{ X \ar@{>}[rr]^{f} \ar[dr]_{\hbox{id}} & & G \\ & \upsilon_{\sA^*}X\ar[ur]_{f^{\gu}} &}
\]
That is, with conditions above, for each  $f\in\sA$  there is a unique continuous extension $f^{\gu}\colon \upsilon_{\sA^*}X\to G$ such that $f^{\gu}|_X=f$.
\elem
\begin{proof}
If $f\in\sA^*$ we define $f^{\gu}=f^b|_{\upsilon_{\sA^*}X}$. Take now $f\in\sA\setminus\sA^*$.
It will suffice to extend $f$ to a continuous function $f^{\gu}\colon X\cup\{p\}\to G$ such that $f^{\gu}|_X=f$ for all $p\in\upsilon_{\sA^*}X\setminus X$.

Since $G$ is discrete and every map in $\sA^*$ has relatively compact range, it follows that every map in $\sA^*$ has finite range.
On the other hand, $f\in\sA$, which means that there is $\{f_i:i<\omega\}\subseteq\sA^*$,
such that $\{coz(f_i): i<\omega\}$ is locally finite and $f=\prod\limits_{i<\omega}f_i$. This means that $f(X)$ is a countable subset of $G$.
That is $f(X)=\{a_n\}_{n<\omega}\subseteq G$. Thus, the family $\{f^{-1}(a_n):n<\omega\}$ consists of disjoint clopen subsets such that
$\bigcup\limits_{n<\omega}f^{-1}(a_n)=X$.
 We claim that there exists a unique $n<\omega$ such that $p\in\overline{f^{-1}(a_n)}$.

\noindent \emph{Existence}: Reasoning by contradiction suppose that for every $n<\omega$ there is an open neighborhood  $U_n$ of $p$ in $\beta_{\sA^*}X$
such that $U_n\cap f^{-1}(a_n)=\emptyset$. Then $p\in\bigcap\limits_{n<\omega}U_n\in\tau_{\sA^*}(\omega)$ and
$\bigcap\limits_{n<\omega}U_n\cap X=\emptyset$, which is a contradiction since $p\in\upsilon_{\sA^*}X$.
\mkp

\noindent \emph{Uniqueness}: 
Since $f^{-1}(a_n)\cap f^{-1}(a_m)=\emptyset$, if $n\not= m$, and $\sA$ is normal, there are $D_n,D_m\in\sD(\sA)$ such that $f^{-1}(a_n)\subseteq D_n$,
$f^{-1}(a_m)\subseteq D_m$
and $ \overline D_n\cap  \overline D_m=\emptyset$, which implies $\overline{f^{-1}(a_n)}\cap \overline{f^{-1}(a_m)}=\emptyset$.

Consider now $X\cup\{p\}$ with the topology inherited from $\beta_{\sA^*}X$  and define $f^\gu(p)=a_n$ and $f^\gu|_X=f$.
We must prove that $f^\gu$ is continuous at $p$.
Since $f^{-1}(a_n)$ is disjoint from $X\setminus f^{-1}(a_n)$ and $\sA$ is normal, there are $D_1,D_2\in\sD$ such that
$(f^{\gu})^{-1}(a_n)=f^{-1}(a_n)\cup \{p\}\subseteq \overline D_1$, $X\setminus f^{-1}(a_n)\subseteq D_2$ and $D_1\cap D_2=\emptyset$.
Since $\overline D_1$ is clopen in $\beta_{\sA^*}X$ and contains $p$, we have $\overline D_b\cap(X\cup\{p\})=(f^{\gu})^{-1}(a_n)$ is clopen in
$X\cup \{p\}$. Therefore $f^\gu$ is continuous at $p$.

\end{proof}

We now remember some definitions concerning abelian groups.

\bdfn
Let $\Z^{\N}$ be the so-called {\em Baer-Specker group\/} \cite{Baer,Specker}, that is, the group of all integer sequences, with pointwise addition.
For each $n<\omega$, let $e_n$ be the sequence with $n$-th term equal to $1$ and all other terms $0$.
According to  {\L}o\'{s} (cf. \cite{Fuchs}), a torsion-free abelian group $G$ is said to be \emph{slender} if every homomorphism from $\Z^\N$ into $G$ maps all but finitely many
of the $e_n$ to the identity element. It is a well known fact that every free abelian group is slender and that
every homomorphism from $\Z^\N$ into a slender group factors through $\Z^n$ for some natural number $n$. (cf. \cite{Dudley,Nunke}).
\edfn
\mkp

We next prove that when the group $G$ is slender and  $\sA$ is a function group in $G^X$ that is the largest $\omega$-extension
of its bounded subgroup $\sA^*$, then support sets are included in $\upsilon_{\sA^*}X$.

\begin{prop}\label{localfinita}
Let $\mathcal{A}$ be a function group in $G^X$ such that $\sA^*$ is controllable.
Let $\phi:\mathcal{A}\rightarrow G$ be a weakly separating group homomorphism such that $\phi|_{\sA^*}\neq e_G$.
\begin{enumerate}[(a)]
\item If either $\phi$ is continuous with respect to the pointwise topology or $G$ is discrete, then
the minimum weak support set for $\phi$ is a singleton ``$s$'' that is a support set for $\phi|_{\sA^*}$.
\item If $G$ is a slender discrete group and $\sA$ is the largest $\omega$-extension of $\sA^*$,
then  $s\in\upsilon_{\sA^*}X$.
\end{enumerate}
\end{prop}
\begin{proof}
(a) Applying Propositions \ref{LEMA 1}, \ref{LEMA 2}, and\ \ref{LEMA 3}, it follows that the minimum weak support for $\phi$ is a singleton ``$s$''
that is a support set for $\phi|_{\sA^*}$.
\medskip

(b) Suppose, reasoning by contradiction, that  $s\notin \upsilon_{\sA^*}X$. Then, there is a subset $C$ which is $\tau(\omega)$-closed containing $X$ but $s\notin C$.
Hence, there exists an strictly decreasing sequence  $\{V_n\}_{n<\omega}$ of open subsets in $\beta_{\sA^*}X$
(that can be assumed also closed if the group $G$ is discrete)
such that $V_{n+1}\subseteq \overline{V}_{n+1}\subseteq V_n$, $s\in \bigcap\limits_{n<\omega} V_n$, and
$\bigcap\limits_{n<\omega} V_n\cap X\subseteq\bigcap\limits_{n<\omega}V_n\cap C=\emptyset$.
As a consequence, it is readily seen that $\{V_n\}_{n<\omega}$ is locally finite.
Now, We are going to obtain a sequence $\{f_n\}_{n<\omega}\subseteq\sA^*$ such that $\{coz(f_n):n<\omega\}$ is locally finite.

Indeed, assume WOLG that  $\beta_{\sA^*}X\setminus V_n\neq \emptyset$ for all $n<\omega$. Since every support set for $\phi$ contain $s$, it follows that
$\beta_{\sA^*}X\setminus V_n$ is a closed subset that is not a support for $\phi$. So, there is $f_n\in\sA^*$ such that $\beta_{\sA^*}X\setminus V_n\subseteq Z(f^b_n)$
and $\phi(f_n)\neq e_G$, for each $n<\omega$.

Take an element $\{a_n\}\in \Z^{\N}$ and set $f_n^{a_n}(x)\defi f_n(x)^{a_n}$ for all $x\in X$. Since $coz(f_n^{a_n})\subseteq coz(f^b_n)\subseteq V_n$,
we have that $\{coz(f_n^{a_n})\}_{n<\omega}$ is locally finite in $X$. 

By hypothesis $\sA$ is the largest $\omega$-extension of $\sA^*$, which implies $\prod\limits_{n<\omega}f_n^{a_n}\in\sA$. Therefore,
we can define the group homomorphism $\psi: \Z^{\N}\colon\to G$\ by
$$ \psi(\{a_n\})\defi \phi(\prod\limits_{n<\omega}(f_n)^{a_n})\in G.$$

In particular, applying the homomorphism $\psi$ to the basic elements $e_n\in \Z^{\N}$,
for all $n<\omega$, we obtain

$\psi (e_n)=\phi(f_n)\neq e_G$ for all $n<\omega$. This is a contradiction since $G$ is a slender group by our initial assumption.
\end{proof}

In the previous section, we have just seen how a weakly separating  homomorphism $H\colon \sA\to \sB$,  where $\sA\subseteq G^X$ is an extension of
a controllable function group $\sA^*$, has associated a continuous (weak support) map $h$ that assigns to each point $y\in Y$
the weak support for $\delta_y\circ H_{|\sA^*}$.
Our next goal is to obtain a complete representation of $H$ by means of the map $h$. In order to do it, we need that
$h(y)$ be a support set for all $y\in Y$ and that $h(Y)\subseteq \upsilon_{\sA^*}X$.
Two conditions assuring these requirements have been given in  Proposition \ref{localfinita}. \mkp

\bdfn\label{peso}
Let $\sA$ and $\sB$ be two function groups in $G^X$ and $G^Y$, respectively, such that
$\sA^*$ is controllable. Given a weakly separating  homomorphism $H\colon \sA\to \sB$,
for each $y\in Y$, define the subgroup of $G$ $$G_{h(y)}\defi \{(f^b\circ h)(y):f\in\sA^*\}$$
and denote by $Hom(G_{h(y)},G)$
the space of all group homomorphisms of $G_{h(y)}$ into $G$, equipped with the pointwise convergence topology.
Consider now the set $$\sG=\bigcup\limits_{y\in Y}Hom(G_{h(y)},G).$$ We can think of all elements of $\sG$ as partial functions on $G$,
that is, functions $\alpha:Dom(\alpha)\subseteq G\to G$ whose domain is a (not necessarily proper) subset of $G$.
We equip $\sG$ with the pointwise convergence topology as follows.

Let $\alpha\in\sG$, $a_1,\cdots,a_n\in Dom(\alpha)$ and $U$ an open neighborhood at $e_G$ in $G$, the set
$[\alpha;a_1,\cdots,a_n,U]=\{\beta\in\sG:\exists b_i\in Dom(\beta), \alpha(a_i)^{-1}\beta(b_i)\in U,\ 1\leq i\leq n\}$  is a basic neighborhood of $\alpha$.
It is easily verified that this procedure extends the pointwise convergence topology over $\sG$. With this notation we define
$$w:Y\to \sG$$ by $$w[y]((f^b\circ h)(y))\defi Hf(y)$$ for each $y\in Y$ and $f\in\sA^*$.
\edfn
\mkp
We shall verify next that, under some mild conditions, the map $w$ is well defined and continuous. Furthermore,
if $G$ is discrete, Lemma \ref{extension} implies that
if $f\in\sA$ and $h(y)\in\upsilon_{\sA^*}X$, then $f^\gu(h(y))\in G_{h(y)}$.

\bprp\label{wcontinua}
Let $\sA$ and $\sB$ be two function groups in $G^X$ and $G^Y$, respectively,
such that $\sA^*$ is controllable. 
Let $H\colon\sA\to\mathcal{B}$ be a weakly separating homomorphism that is
either continuous for the pointwise convergence topology or, otherwise, assume that $G$ is a (discrete) slender group. Then
the following assertions hold:
\begin{enumerate}[(a)]
\item $w(y)$ is a well defined group homomorphism for all $y\in Y$.
\item $w$ is continuous if $\sG$ is equipped with the pointwise convergence topology.
\end{enumerate}
\eprp

\bpf
(a) Let $h:Y\to\beta_{\sA^*}X$ be the weak support map canonically associated to $H$.
In order to prove that $w(y)$ is well defined, take $f,g\in \sA^*$ such that $f^b(h(y))=g^b(h(y))$.
Then, by Proposition \ref{PROPOSICION 1}.1, $(g^{-1}f)^b(h(y))=e_G$. By the way it was defined, we have that $h(y)$ is a weak support for $\delta_y\circ H$.
On the other hand, Proposition \ref{localfinita} implies that $\{h(y)\}$ is actually a support subset for
$\delta_y\circ H$. Hence $H(g^{-1}f)(y)=e_G$ or, equivalently, $Hf(y)=Hg(y)$. Therefore $w(y)(f^b(h(y)))=Hf(y)=Hg(y)=w(y)(g^b(h(y)))$.
The verification that $w(y)$ is a group homomorphism is straightforward.

(b) Let $(y_d)_{d\in D}$ be a net converging to $y$ in $Y$. Take $a_1,\cdots,a_n\in Dom(w(y))=G_{h(y)}$ and $U$ an open neighborhood of $e_G$ in $G$.
Then there is $f_i\in\mathcal{A}^*$ such that $a_i=(f_i^b\circ h)(y)$, $1\leq i\leq n$. That is $w(y)(a_i)=Hf_i(y)$ and, since $Hf_i$ is continuous,
there is an index $d_i$ such that $Hf_i(y_d)\in Hf_i(y)U$, for all $d\geq d_i$, $1\leq i\leq n$. Let $d\geq d_1,\cdots,d_n$,
then $b_i^{(d)}=(f^b_i\circ h)(y_d)\in G_{h(y_d)}$ and $w(y)(a_i)^{-1}w(y_d)(b_i^{(d)})=(Hf_i(y))^{-1}Hf_i(y_d)\in U$, $1\leq i\leq n$, that is, $w(y_d)\in[w(y);a_1\cdots,a_n,U]$.
\epf
\mkp

We assume below that $\upsilon_{\sA^*}X=X$
for certain function groups $\sA$. It is easily seen that this happens when $(X,\tau_\sA)$ is Lindel\"of.
For example, if the set $X$ is countable or $\sigma$-compact in the topology $\tau_{\sA}$.

\bthm\label{teocaso1} Let $\sA$ and $\sB$ be two function groups in $G^X$ and $G^Y$, respectively, such that
$\sA^*$ is controllable, $\upsilon_{\sA^*}X=X$, $\upsilon_{\sB^*}Y=Y$ and $\sA$ is the largest $\omega$-extension of $\sA^*$.
Let $H\colon\sA\to\mathcal{B}$ be a weakly separating homomorphism that is
either continuous for the pointwise convergence topology or, otherwise, assume that $G$ is a (discrete) slender group.
Then there are continuous maps $$h\colon Y\to X$$ and $$w \colon Y\to\bigcup_{y\in Y} Hom(G_{h(y)},G)$$ satisfying the following properties:
\begin{enumerate}[(a)]
\item For each $y\in Y$ and every $f\in\sA$ it holds $Hf(y)=w[y](f(h(y))]$.
\item If $G$ is a discrete group, then $H$ is continuous with respect to the pointwise convergence topology.
\item If $G$ is a discrete group, then $H$ is continuous with respect to the compact open topology.
\item If $H$ is weakly biseparating and $\sB$ is the largest $\omega$-extension of $\sB^*$, which is controllable, then $h$ is a homeomorphism.
\end{enumerate}
\ethm
\begin{proof}
First, let $h:Y\to\beta_{\sA^*}X$ be the continuous weak support map canonically associated to $H$.
By Propositions \ref{LEMA 1} and \ref{localfinita}, we have
that $h(Y)\subseteq \upsilon_{\sA^*}X=X$ by our initial assumptions.
On the other hand, Proposition \ref{wcontinua}
yields the continuity of the weight map $w$.

\noindent Now, we prove the properties formulated above:

\noindent (a) is straightforward after the definition of $w$.

\noindent (b) Assume that $G$ is discrete since there is nothing to prove otherwise. Then every pointwise convergent net $(f_i)$ in $\sA$
must be pointwise eventually constant. That is, for every $x\in X$, there is $i_0$ such that $i\geq i_0$ implies that
$f_i(x)=f_{i_0}(x)$. Thus the continuity of $H$ follows from (a).

\noindent (c)\ Let $(f_d)_d\subseteq \sA$ be a net converging to the constant function $e_G$ in the compact open topology.
If $K$ is a compact subset of $Y$, then $h(K)$ is a compact subset in $X$ by the continuity of $h$.
Therefore $(f_d)_d$ is eventually the constant function $e_G$ on $h(K)$. Applying (a), it follows that
$(Hf)_d$ is eventually $e_G$ on $K$, which completes the proof.\mkp
\mkp

(d) We have two continuous weak support maps $h^b\colon \beta_{\sB^*}Y\to \beta_{\sA^*}X$ and
$k^b\colon  \beta_{\sA^*}X\to \beta_{\sB^*}Y$ associated to $H|_{\sA^*}$ and $(H^{-1})|_{\sB^*}$, respectively.
Let us to see that $(h^b\circ k^b)|_{X}=\text{id}_X$. First, notice that by Propositions \ref{LEMA 1} and \ref{localfinita},
we have $k(x)\in\gu_{\sB^*}Y=Y$ for all $x\in X$.
Therefore we have $h^b(k(x))$ is actually $h(k(x))$

Reasoning by contradiction, suppose there is $x\in X$ such that $x\neq h(k(x))$.
By Proposition \ref{PROPOSICION 1}.7, there are $D_1,D_2\in\sD$
such that $x\in int(\overline{D}_1)$, $h(k(x))\in int(\overline{D}_2)$ and $D^b_1\cap D^b_2=\emptyset$. Take $f_1\in \sA^*$ such that $f_1(x)\neq e_G$.
Applying that $\sA^*$ is controllable to $f_1$, $D_1$ and $D_2$, there are $U_1\in\sE(\sA)$ and $f'_1\in\sA^*$ such that $D_1\subseteq U_1$,
$U_1\cap D_2=\emptyset$, $f'_1|_{D_1}=f_1|_{D_1}$ and $f'_1|_{Z(f_1)\cup(X\setminus U_1)}=e_G$. Therefore $x\in coz(f'_1)\subseteq U_1$.
Applying assertion (a) to $H^{-1}$ and $x$, we obtain  $$e_G\not=f'_1(x)=H^{-1}Hf'_1(x)=\rho[y](Hf'_1(k(x))]$$ for an associated map $\rho$.
Since $\rho[y]$ is a group homomorphism, this implies that $Hf'_1(k(x))\not=e_G$ and, as a consequence, $k(x)\in coz(Hf'_1)$.

On the other hand $h(k(x))\in int(\overline{D}_2)\subseteq int(\overline{Z(f'_1)}$ and, since
$h(k(x))$ is a weak support set for $\delta_{k(x)}\circ H$, we obtain that $Hf'_1(k(x))=e_G$.
This is a contradiction that completes the proof.

In like manner, we can prove that $k\circ h=\text{id}_Y$. Therefore $k=h^{-1}$ and $h$ is a homeomorphism.
\end{proof}
\mkp

Recall that a function group $\sA$ in $G^X$ is $\sA$ is \emph{pointwise dense} if
$\delta_x(\sA)$ is dense in $G$ for all $x\in X$.

\bcor\label{corcaso1} Let $\sA$ and $\sB$ be two function groups in $G^X$ and $G^Y$, respectively, such that $\sA$ is pointwise dense,
$\sA^*$ is controllable, $\upsilon_{\sA^*}X=X$, $\upsilon_{\sB^*}Y=Y$ and $\sA$ is the largest $\omega$-extension of $\sA^*$.
Let $H\colon\sA\to\mathcal{B}$ be a weakly separating homomorphism that is
either continuous for the pointwise convergence topology or, otherwise, assume that $G$ is a (discrete) slender group.
Then there are continuous maps $$h\colon Y\to X$$ and $$w \colon Y\to End(G)$$ satisfying the following properties:
\begin{enumerate}[(a)]
\item For each $y\in Y$ and every $f\in\sA$ it holds $Hf(y)=w[y](f(h(y))]$.
\item If $G$ is a discrete group, then $H$ is continuous with respect to the pointwise convergence topology.
\item If $G$ is a discrete group, then $H$ is continuous with respect to the compact open topology.
\item If $H$ is weakly biseparating and $\sB$ is the largest $\omega$-extension of $\sB^*$, which is controllable, then $h$ is a homeomorphism.
\end{enumerate}
\ecor
\mkp

We are in position of establishing the main result in this paper.

\bthm\label{teoremacaso1}
Let $\sA$ and $\sB$ be two, pointwise dense, function groups in $G^X$ and $G^Y$, respectively, such that
$\sA^*$ and $\sB^*$ are controllable, $\upsilon_{\sA^*}X=X$, $\upsilon_{\sB^*}Y=Y$, and $\sA$ and $\sB$
are the largest $\omega$-extensions of their bounded subgroups.
Let $H\colon\sA\to\mathcal{B}$ be a weakly biseparating isomorphism that is
either continuous for the pointwise convergence topology or, otherwise, assume that $G$ is a (discrete) slender group.
Then the function groups $\sA$ and $\sB$ are equivalent.

\noindent That is, there are continuous maps $$h\colon Y\to X$$
and $$w \colon Y\to Aut(G)$$ satisfying the following properties:
\begin{enumerate}[(a)]
\item $h$ is a homeomorphism.
\item For each $y\in Y$ and every $f\in\sA$ it holds $Hf(y)=w[y](f(h(y))]$.
\item $H$ is a continuous isomorphism with respect to the pointwise convergence topology.
\item $H$ is a continuous isomorphism with respect to the compact open topology.
\end{enumerate}
\ethm
\bpf
We only need to verify that $w[y]\in Aut(G)$ for all $y\in Y$. Applying Corollary \ref{corcaso1} to $H^{-1}$, we obtain maps $$\rho\colon Y\to End(G)$$ and
$$k\colon X\to Y$$ such that for every $x\in X$ and $g\in\sB$, we have $H^{-1}g(x)=\rho[x](g(k(x)))$. Thus, for every $f\in\sA$ and $x\in X$, we have
$$f(x)=H^{-1}\circ(Hf)(x)=\rho[x](Hf(k(x)))=\rho[x](w[k(x)](f(h(k(x))))$$  and $h^{-1}=k$ by the proof of Theorem \ref{teocaso1}(d).
Therefore $$f(x)=H^{-1}\circ (Hf)(x)=\rho[x](Hf(k(x)))=\rho[x](w[k(x)](f(x))$$ and
$$g(y)=H\circ (H^{-1}g)(y)=w[y](H^{-1}g(h(y)))=w[y](\rho[h(y)](g(y)).$$ Applying the former equality to $x=h(y)$,
it follows that $\rho[h(y)]\circ w[y]=\text{id}_G$ for all
$y\in Y$, and from the latter, we also have that $w[y]\circ\rho[h(y)]=\text{id}_G$. This means that $w[y]$ is an automorphism on $G$.
\epf

\textbf{Acknowledgment:} The authors thank the referees for several helpful comments.

\end{document}